\documentclass[12pt,a4paper]{amsart}
\pdfoutput=1
\usepackage{amsmath,amsfonts,amssymb,amsthm}  
\usepackage[utf8]{inputenc}
\theoremstyle{plain}
\newtheorem{theorem}{Theorem}[section]

\newtheorem{lemma}[theorem]{Lemma}

\theoremstyle{definition}
\newtheorem{definition}[theorem]{Definition}

\newtheorem{remark}[theorem]{Remark}

\numberwithin{equation}{section}
\newtheorem*{theorem*}{Theorem}

\def\Xint#1{\mathchoice
{\XXint\displaystyle\textstyle{#1}}
{\XXint\textstyle\scriptstyle{#1}}
{\XXint\scriptstyle\scriptscriptstyle{#1}}
{\XXint\scriptscriptstyle\scriptscriptstyle{#1}}
\!\int}
\def\XXint#1#2#3{{\setbox0=\hbox{$#1{#2#3}{\int}$}
\vcenter{\hbox{$#2#3$}}\kern-.5\wd0}}

\def\dashint{\Xint-}

\newcommand{\dmu}{\, \mathrm{d}\mu}
\newcommand{\dx}{\, \mathrm{d}x}

\DeclareMathOperator{\BMO}{BMO}

\providecommand{\abs}[1]{ \lvert#1  \rvert}
\providecommand{\norm}[1]{ \lVert#1  \rVert}

\newcommand{\art}[6]{{\sc #1, \rm #2, \it #3 \bf #4 \rm (#5), \mbox{#6}.}}
\newcommand{\book}[3]{{\sc #1, \it #2, \rm #3.}}
\newcommand{\AND}{{\rm and }}

\title[Homeomorphisms of the Heisenberg group preserving
$\BMO$]{Homeomorphisms of the Heisenberg group preserving $\BMO$}

\author{Riikka Korte}
\address[R.K]{University of Helsinki,
  Department of Mathematics and Statistics,
  P.O. Box 68, FI-00014 University of Helsinki, Finland}
\email{riikka.korte@helsinki.fi}

\author{Niko Marola}
\address[N.M.]{University of Helsinki,
  Department of Mathematics and Statistics,
  P.O. Box 68, FI-00014 University of Helsinki, Finland}
\email{niko.marola@helsinki.fi}

\author{Olli Saari}
\address[O.S.]{Aalto University,
  Department of Mathematics and Systems Analysis,
  P.O. Box 11100, FI-00076 Aalto, Finland}
\email{olli.saari@aalto.fi}

\date{}

\thanks{The research is supported by the Academy of Finland and the V\"ais\"al\"a Foundation.}
\subjclass[2010]{30L10, 42B35}
\keywords{Carnot group, Function of bounded mean oscillation, Heisenberg group, Metric space, Quasiconformal mapping, Stratified group.}

\begin{document}

\begin{abstract}
 We provide a new geometric proof of Reimann's theorem characterizing quasiconformal mappings as the ones preserving functions of bounded mean oscillation. While our proof is new already in the Euclidean spaces, it is applicable in Heisenberg groups as well as in more general
  stratified nilpotent Carnot groups.
\end{abstract}

\maketitle

\section{Introduction}

It is well known that composition with quasiconformal mappings of the
Euclidean spaces preserves functions of bounded mean oscillation, and
with an additional differentiability assumption the preservation property
implies that the mapping is quasiconformal. This result was proved by
Reimann~\cite{Re-qc} in $\mathbb{R}^n$, $n\geq 2$. The corresponding
problem on the real-line was solved by Jones~\cite{Jones}. Later
analogous results were proved for maximal functions and Muckenhoupt
weights, see Uchiyama~\cite{Uchiyama}. 

Astala~\cite{Astala} showed that by localizing the preservation property, the differentiability
assumptions could be removed.  Using the approach initiated by Astala, Staples~\cite{Staples-ainfty, Staples-doubling} proved analogues of Uchiyama's theorems, and finally, following Astala's idea of localizing the preservation property, Vodop'yanov and Greshnov generalized the Reimann type characterization of quasiconformal mappings to Carnot groups and Carnot--Carath\'eodory spaces in \cite{VoGr}. However, even in Euclidean spaces, it remains open whether preservation of $\BMO$ alone implies quasiconformality. The known results assume either differentiability or \textit{local} $\BMO$ invariance.

In this paper, we take another, more geometric, point of view and give a new proof of Reimann's theorem that applies in the Heisenberg group. Reimann's original proof relied on a construction of suitable functions of bounded mean oscillation, whereas our approach makes use of a general characterization of $\BMO$-preserving mappings due to Gotoh \cite{Gotoh}, or more precisely its generalization to doubling metric measure spaces by Kinnunen et al.~\cite{KiKoMaSh}. We do not assume local $\BMO$ invariance, but we will use assumptions more close to Reimann's original approach instead. We mention that our proof of Reimann's theorem can also be applied in more general Carnot groups, we refer the reader to the discussion in Remark~\ref{rmk}. But more importantly, it provides a new simple proof in the Euclidean setting with obvious changes.

\section{Preliminaries}

We denote by $\mathbb{H}^{n}$ the $n$th Heisenberg group, that is, the
set $\mathbb{C}^{n} \times \mathbb{R}$ endowed with the group
operation
\[(z,t)(z',t') = ( z+z', t+t'+ 2 {\rm Im} \, \sum_j z_j \overline{z_j'} ) .\]
The inverse of $(z,t)$ is $(z,t)^{-1} = (-z,-t)$. For $z \in \mathbb{C}$, we denote ${ \rm Re} \, z = x$ and ${\rm Im }\, z = y$. The horizontal distribution of the Heisenberg group is spanned by the vector fields
\[X_j = \frac{\partial}{\partial x_j} + 2y_j \frac{\partial}{\partial
  t} \quad \textrm{and} \quad Y_j = \frac{\partial}{\partial y_j} -
2x_j \frac{\partial}{\partial t} \] as $j = 1,\ldots, n$, and the only
non-trivial commutator relation of the tangent bundle is $[X_j,Y_j] =
-4 T$ where $T = \partial/\partial t$ is the vertical direction.

A path $\gamma : [0,1] \to \mathbb{H}^{n}$ is horizontal if
$\gamma'(t)$ is horizontal for all $t$. We define the
Carnot--Carath\'eodory distance as the length of the shortest
horizontal path joining two points in $\mathbb{H}^n$, and it will be
denoted by $d$. Using this metric, the left-translations are
isometries of the Heisenberg group, i.e. $d(p,q) = d(lp,lq)$, where
$p,q,l \in \mathbb{H}^{n}$.

The Lebesgue measure on $\mathbb{R}^{2n+1}$ is the bi-invariant Haar
measure for $\mathbb{H}^{n}$, and it will be denoted by $\abs{ \cdot
}$. Moreover, we define the dilations by positive reals $\delta$ as
$\delta (z,t) = (\delta z, \delta ^{2} t)$. The dilations satisfy
$\abs{\delta E} = \delta^{2n+2} \abs{E}$ and $d (\delta z,\delta z') =
\delta d (z,z')$.

There is another metric on $\mathbb{H}^{n}$ that is bi-Lipschitz
equivalent to the Carnot--Carath\'eodory metric, namely, the one induced
by the Koranyi norm. For $p=(z,t) \in \mathbb{H}^{n}$, we define
\[d_{K}(p,0) = (\abs{z}^{4} + \abs{t}^{2})^{1/4}.\]
Hence the metric notions can be defined using whichever metric.

A homomorphism $L : \mathbb{H}^{n} \to \mathbb{H}^{n}$ is homogeneous
if it commutes with dilations. A mapping $f: \mathbb{H}^{n} \to
\mathbb{H}^{n}$ is said to be Pansu differentiable at $p \in
\mathbb{H}^{n}$ if there is a homogeneous homomorphism $L$ such that
\[ \frac{d(f(p)^{-1}f(hp), L(h) )}{d(h,0)} \to 0 \] as $h \to 0$. The
homogeneous homomorphism $L$ satisfying the limit is unique and is
called Pansu differential of $f$ at $p$. Note that the homogeneous
homomorphism $L$ is continuous if $f$ is. We refer the reader, for example,
to \cite{Magnani} for a detailed discussion on differential calculus
on general stratified groups.

Some results in this paper are most conveniently stated in more
general setting of doubling metric measure spaces. The Heisenberg
group and more general Carnot groups discussed in Remark~\ref{rmk},
however, are known to be particular examples of such general
spaces. For brevity, we refer the reader to \cite{Heinonen} and
\cite{Bjorn} for this concept of a doubling metric measure space
$(X,d,\mu)$ and also for all the necessary definitions.

A locally integrable function $u$ on a metric measure space with a
doubling measure $\mu$ is said to be of bounded mean
oscillation, abbreviated to $\BMO(X)$, if its mean oscillations on metric
balls $B\subset X$ are uniformly bounded, that is,
\[ 
\sup_{B} \dashint_{B} \abs{u-u_{B}} \dmu =: \norm{u}_{\BMO} < \infty.
\]
Here both $u_B$ and the barred integral mean the integral average of
$u$ over a ball $B$.  
%The class $\BMO$ is very common in Harmonic Analysis and PDE. 
One of the most fundamental properties of the class $\BMO$ is the
following John--Nirenberg lemma. For a proof, we refer the reader to \cite{Bjorn}.

\begin{lemma} 
\label{lemma:JN}
Let $X$ be a metric measure space with a doubling measure $\mu$. Suppose that $u\in \BMO(5B)$. Then for every $\lambda>0$
\[
\mu(\{x\in B\,:\,|u(x)-u_{B}|>\lambda\})\leq
2\mu(B)\exp\left(-\frac{A\lambda}{\|u\|_{\BMO}}\right),
\]
where $A$ depends only on the doubling constant of the measure $\mu$.
\end{lemma}

A homeomorphism of metric spaces $f: X \to X $ respecting null sets is
said to be $\BMO$-preserving if $\norm{u \circ f^{-1}}_{\BMO} \leq C_f
\norm{u }_{\BMO}$. This property has a characterization in terms of
densities, and the following theorem was first found by Gotoh in
\cite{Gotoh}. The metric measure space version of the theorem can be
found in \cite{KiKoMaSh}.

\begin{theorem} 
\label{thm:gotoh}
A homeomorphism $f : X \to X$ on a metric measure space with a
doubling measure $\mu$ such that $f^{-1}(E)$ is a null set for all
null sets $E \subset X$ is $\BMO$-preserving if and only if the
following holds: For any pair of measurable sets $E_1 , E_2$ of $X$ we
have
\[
\sup_{B} \min _{i} \frac{\mu(E_i \cap B)}{\mu(B)} \leq K \left(
  \sup_{B} \min_{i} \frac{\mu(f(E_i) \cap B)}{ \mu(B) }
\right)^{\alpha}\] for universal constants $\alpha ,K >0$. The
supremum is taken over all metric balls in $X$.
\end{theorem}

Given a homeomorphism $f:X \to X$, then for every $x\in X$ and $r>0$ set
\[
K_f(x,r) =  \frac{\sup\{d(f(x),f(y)):\, d(x,y)=r\}}{\inf\{d(f(x),f(y)):\ d(x,y)=r\}}.
% \frac{\sup\{d(x,y): x,y \in f(B(z,r)) \}}{\inf\{d(x,y): x,y \in
% f(B(z,r)) \}} \leq K<\infty
\] 
We recall that $f$ is called quasiconformal if 
\[
\limsup_{r\to 0}K_f(x,r) \leq K
\]
for all $x \in X$ and for some uniform $K<\infty$. A homeomorphism $f$
is quasi-symmetric, if the inequality above is satisfied without
$\limsup$, that is, it is satisfied for all $x\in X$ and all $r>0$. We
refer the reader to \cite{HeKoInv, HeKoActa, HKST} for more on
quasiconformal mappings between metric spaces.

\section{Characterization via $\BMO$-maps}

In this section we prove that $\BMO$ preserving maps are
quasiconformal and vice versa, provided that one assumes some
additional regularity. The proof of the necessity part is new, and it
shows that Gotoh's characterization of $\BMO$-maps is one of the
underlying phenomena connecting quasiconformal and $\BMO$-preserving
maps. The following is our main result, see also Remark~\ref{rmk}.

\begin{theorem}
\label{thm:equivalence}
Let $f:\mathbb{H}^{n} \to \mathbb{H}^{n}$ be an almost everywhere
Pansu differentiable homeomorphism such that $f^{-1}(E)$ is a null set
for all null sets $E$. Then $f$ is quasiconformal if and only if it is
$\BMO$-preserving.
\end{theorem}

The proof of Theorem~\ref{thm:equivalence} is divided into
Lemma~\ref{prop:quasi implies bmo} and Lemma~\ref{lemma2}. The proof
is new already in the Euclidean spaces, see \cite{Re-qc}.

\subsection{Sufficiency}
The sufficiency part is most conveniently stated in a much more
general setting than the Heisenberg group. It is certainly well known
among specialists that this part is true but nevertheless it seems to
lack references. We state it in full generality and note that the
Heisenberg group satisfies the assumptions of the following
lemma. This fact and all the necessary definitions can be checked in
\cite{HeKoActa, HKST}.

\begin{lemma}
\label{prop:quasi implies bmo}
Let $X$ be a linearly locally connected unbounded metric measure space
of $Q$-bounded geometry. Suppose that $f: X\to X$ is $K$-quasiconformal
and $u\in \BMO(X)$. Then
\[
\| Fu \|_{\BMO} \leq C\|u\|_{\BMO},
\]
where we write $Fu=u\circ f^{-1}$ and $C$ is a positive constant depending only on $K$ and the data of $X$.
\end{lemma}

\begin{proof}
  By \cite{HeKoActa} (see also \cite[Theorem 9.10]{Heinonen}) $f$ is
  quasisymmetric, and thexpull-back measure $\mu(f(E))$ and $\mu$ are known to be $A_\infty$-related. As a
  consequence of quasisymmetry (see \cite[Proposition 10.8]{Heinonen},
  for each ball $B'\subset f(X)\subseteq X$ there is a ball $B\subset
  X$ such that $B'\subset f(B)$ and $\mu(f(B)) \leq C\mu(B')$ for some
  positive constant $C$.

By Lemma~\ref{lemma:JN},
\begin{align*}
  \dashint_{B'}|Fu-&u_B|\, \dmu  = \frac{1}{\mu(B')}\int_0^\infty\mu(\{x\in B':\ |Fu(x)-u_B|>\lambda\})\, d\lambda \\
& \leq  \frac{C}{\mu(f(B))}\int_0^\infty\mu(\{x\in f(B):\ |Fu(x)-u_B|>\lambda\})\, d\lambda \\ 
& \leq \frac{C}{\mu(B)^\delta}\int_0^\infty\mu(\{x\in B:\ |u(x)-u_B|>\lambda\})^\delta\, d\lambda \\
  & \leq C\int_0^\infty \exp\left(-\delta
    A\lambda/\|u\|_{\BMO}\right)\, d\lambda \leq C\|u\|_{\BMO},
\end{align*}
where $C$ is independent of $B$. It is straightforward to verify that 
\[
\dashint_{B'}|Fu-(Fu)_{B'}|\, \dmu \leq 2\dashint_{B'}|Fu-u_B|\dmu,
\]
and hence the claim follows.
\end{proof}

\subsection{Necessity} 
The proof of the necessity part proceeds through approximation by
linear mappings. In order to do that, one has to assume 
differentiability. In Heisenberg group, we have to assume Pansu
differentiability in order to approximate by homogeneous
homomorphisms. 

Reimann's original proof in \cite{Re-qc} constructed an explicit
$\BMO$-function such that a differentiable map keeping it in $\BMO$
had to be quasiconformal. Our proof is less constructive, and one
motivation for it is to point out that the interaction between
$\BMO$-preserving and quasiconformality is closely related to
densities. Indeed, both classes of mappings have a characterization in
terms of densities. For the quasiconformal case, we refer the reader
to \cite{KoMaSh} and \cite{GeKe}. For the $\BMO$-maps the
characterization is stated in Theorem~\ref{thm:gotoh}. These
characterizations are, however, slightly different.

Our proof will work in Euclidean spaces and also in more general
Carnot groups as is discussed in
Remark~\ref{rmk}. %and it is in some sense easier to follow than the
                  %original constructive proof (provided that the
                  %reader already knows the results about $\BMO$ used
                  %here).

\begin{lemma} \label{lemma2} Let $f: \mathbb{H}^{n} \to
  \mathbb{H}^{n}$ be an almost everywhere Pansu differentiable
  $\BMO$-preserving homeomorphism such that $f^{-1}(E)$ is a null set
for all null sets $E$. Then $f$ is quasiconformal.
\end{lemma}

\begin{proof}
  Take any point of differentiability $z \in \mathbb{H}^{n}$ and let
  $B(z,r)$ be a ball centred at it. Since left-translations are
  isometries, we may assume that $z = 0$. Let $L$ be the homogeneous
  homomorphism given by Pansu differentiability. For any points $x,y
  \in \mathbb{H}^{n}$ there is $v \in \partial B(0,1)$ such that
\begin{align*}
d(L(x),L(y)) &= d( L(y)^{-1}L(x),0) =  d(L(y^{-1}x),0) \\
& = d(y^{-1}x,0) d(L(v),0)  
= d(x,y) d(L(v),0).
\end{align*} 
We denote by $\lambda_{\max}$ the maximum that the functional $v \mapsto d(L(v),0)$ attains in $\partial B(0,1)$. Recall that $L$ is continuous.

Take $x \in \partial B(0,\frac{15}{16}r)$ such that  
\[
d(L(x),L(0)) = \lambda_{\max} d(x,0) = \frac{15}{16}r \lambda_{\max}.
\]
%This is seen by using the coordinates: $L(x^{-1}) = L(x)^{-1}$ and since the inverse is computed by reflecting with respect to $0$, we get the result. 
Let $E_1 = B(x,\frac{1}{16}r)$ and $E_2 = B(0,\frac{1}{16}r)$, and take $a \in E_1, b \in E_2$. Then
\begin{align*}
  d(L(a),L(b)) & \geq d(L(x),L(0)) - d(L(x),L(a)) - d(L(0),L(b)) \\
  &\geq \frac{15}{16}r \lambda_{\max} - \frac{1}{16} r \lambda_{\max} -
  \frac{1}{16} r \lambda_{\max} > \frac{3}{4} r \lambda_{\max}.
\end{align*}
Hence 
\[
{\rm dist }(L(E_1),L(E_2)) \geq \frac{3}{4}r \lambda_{\max} \geq \frac{3}{8} {\rm diam} \, L(B(0,r)). 
\]

In Gotoh's Theorem \ref{thm:gotoh}, we may choose $B(0,r)$ to be the ball in the left hand side in order to obtain a lower bound for the supremum. The balls in the right hand side must meet both $E_1$ and $E_2$, so their radii have to exceed ${\rm dist }(E_1,E_2)$. Altogether we get 
\begin{align*}
  1 &\lesssim \sup_{B} \min _{i} \frac{\abs{ E_i \cap B} }{\abs{B}} \leq K \left( \sup_{B} \min_{i} \frac{\abs{L(E_i) \cap B}}{ \abs{B} } \right)^{\alpha} \\
  & \lesssim \left( \frac{ \abs{L(B(0,r))}}{({\rm dist
      }(E_1,E_2))^{2n+2} } \right)^{\alpha} \lesssim \left(
    \frac{\abs{L(B(0,r))}}{({\rm diam}\, L(B(0,r)))^{2n+2}}
  \right)^{\alpha}.
\end{align*}
This estimate holds for all $r > 0$. Letting $r \to 0$, we get the
same estimate for $f$, because near origin $f$ is $L$ up to an
epsilon. So
\[1 \lesssim \liminf_{r \to 0} \frac{\abs{f(B(0,r))}}{({\rm diam}\,
  f(B(0,r)))^{2n+2}} . 
\] 
and it follows by standard arguments that
$f$ is quasiconformal.
\end{proof}

\begin{remark} \label{rmk} It is easy to check that
  Theorem~\ref{thm:equivalence} is valid also in more general Carnot
  groups $\mathbb{G}$ (of step $k$). By these groups we mean simply
  connected Lie groups whose Lie algebra $\mathfrak{g}$ admits a
  nilpotent stratification up to step $k\geq 2$, i.e. $\mathfrak{g} =
  V_1\oplus V_2\oplus\cdots\oplus V_k$ and $[V_1,V_j]=V_{j+1}$ for
  $j=1,\ldots,k-1$ and $[V_1,V_k]=\{0\}$.

  A Carnot group is equipped with a family of non-isotropic dilations
  $\delta_\lambda:\mathfrak{g}\to\mathfrak{g}$ defined as
  $\delta_\lambda(\xi) = \lambda^j$ whenever $\xi\in V_j$,
  $j=1,\ldots,k$, where $\lambda$ is a positive real. The metric
  structure on $\mathbb{G}$ is given by the Carnot--Carath\'eodory
  distance $d$ for which $d(\delta_\lambda z,\delta_\lambda z') =
  \lambda d(z,z')$ for every $z,z'\in\mathbb{G}$. Denoting by $dg$ the
  bi-invariant Haar measure on $\mathbb{G}$, obtained by lifting via
  exponential map the Lebesgue measure on $\mathfrak{g}$, we see that
  $(d\circ \delta_\lambda)(g) = \lambda^Qdg$, where $Q=
  \sum_{j=1}^kj\dim(V_j)$ is the homogeneous dimension of
  $\mathbb{G}$. We refer the
  reader to \cite{Magnani} for more on calculus on stratified groups.
\end{remark}

\end{document}